\documentclass[oneside,english]{article}
\usepackage[T1]{fontenc}
\usepackage{amsthm}
\usepackage{amssymb}
\usepackage[all]{xy}
\usepackage{amsmath}
\usepackage{comment}

\usepackage{color}
\usepackage{mathtools}
\usepackage{tikz-cd}
\tikzcdset{scale cd/.style={every label/.append style={scale=#1},
		cells={nodes={scale=#1}}}}
\usepackage{soul}
\usepackage{xcolor}
\usepackage[many]{tcolorbox}
\tcbset{
	highlight math style={
		colback=yellow,
		arc=0pt,
		outer arc=0pt,
		boxrule=0pt,
		top=2pt,
		bottom=2pt,
		left=2pt,
		right=2pt,
	}
}
\usepackage{framed}
\definecolor{shadecolor}{RGB}{241,231,64}

\makeatletter
\newcommand{\nc}{\newcommand}
\newtheorem{thm}{Theorem}
\theoremstyle{plain}
\nc{\bthm}{\begin{thm}} \nc{\ethm}{\end{thm}}
\newtheorem{prop}[thm]{Proposition}
\nc{\bprp}{\begin{prop}} 
	\nc{\eprp}{\end{prop}}
\newtheorem{ques}[thm]{Question}
\newtheorem{fact}[thm]{Fact}
\nc{\bfct}{\begin{fact}} \nc{\efct}{\end{fact}}
\newtheorem{prob}[thm]{Problem}
\nc{\bprb}{\begin{prob}} \nc{\eprb}{\end{prob}}
\newtheorem{lem}[thm]{Lemma}
\nc{\blem}{\begin{lem}} \nc{\elem}{\end{lem}}
\newtheorem{claim}[thm]{Claim}
\nc{\bclm}{\begin{claim}} \nc{\eclm}{\end{claim}}
\newtheorem{cor}[thm]{Corollary}
\nc{\bcor}{\begin{cor}} \nc{\ecor}{\end{cor}}
\newtheorem{conj}[thm]{Conjecture}
\nc{\bcnj}{\begin{conj}} \nc{\ecnj}{\end{conj}}
\theoremstyle{definition}
\newtheorem{defn}[thm]{Definition}
\nc{\bdfn}{\begin{defn}} \nc{\edfn}{\end{defn}}
\newtheorem{observation}[thm]{Observation}
\nc{\bobs}{\begin{observation}} \nc{\eobs}{\end{observation}}
\theoremstyle{remark}
\newtheorem{rem}[thm]{Remark}
\nc{\brem}{\begin{rem}} \nc{\erem}{\end{rem}}
\newtheorem{cnv}[thm]{Convention}
\nc{\bcnv}{\begin{cnv}} \nc{\ecnv}{\end{cnv}}
\newtheorem{exam}[thm]{Example}
\nc{\bexm}{\begin{exam}} \nc{\eexm}{\end{exam}}

\nc{\bpf}{\begin{proof}} \nc{\epf}{\end{proof}}
\nc{\be}{\begin{enumerate}}
	\nc{\ee}{\end{enumerate}}
\nc{\bi}{\begin{itemize}}
	\nc{\itm}{\item}
	\nc{\ei}{\end{itemize}}
\nc{\invlim}{\lim_{\leftarrow}}
\nc{\dirlim}{\lim_{\rightarrow}}
\nc{\mm}{\mathbf{m}}
\nc{\nn}{\mathbf{n}}
\nc{\FF}{\mathcal{F}}
\nc{\CC}{\mathcal{C}}
\nc{\Span}{\operatorname{span}}
\nc{\Img}{\operatorname{Im}}
\nc{\rank}{\operatorname{rank}}
\nc{\proj}{\operatorname{proj}}
\nc{\F}{\mathbb{F}}
\nc{\Z}{\mathbb{Z}}
\nc{\Q}{\mathbb{Q}}
\nc{\Br}{\operatorname{Br}}

\begin{document}
	\title{Krull-Schmidt Theorem for small profinite groups}
	\author{Tamar Bar-On and Nikolay Nikolov}
	\maketitle 
	\begin{abstract}
		We prove that every small profinite group can be decomposed into a direct product of indecomposable profinite groups, and that such a decomposition is unique up to order and isomorphisms of the components. We also investigate the cancellation property of some free pro-$\mathcal{C}$ groups, and give a new criterion for a profinite group to be small.
	\end{abstract}
	\section{Introduction}
	We start by recalling the Krull-Schmidt Theorem for abstract groups (for a proof see for example \cite[Theorem 6.36]{rotman2012introduction}):
	\begin{thm}[Krull-Schmidt Theorem] \label{kst}
		Let $G$ be an abstract group which satisfies both the Ascending and the Descending chain condition on normal subgroups. Then there exist finitely many indecomposable abstract groups $A_1,...,A_n$ such that $G\cong A_1\times \cdots A_n$. Moreover, if $G\cong B_1\times \cdots \times B_m$ for indecomposable groups $B_1,...,B_n$, then $n=m$ and up to reindexing of the groups $B_1,...,B_n$,  $A_i\cong B_i$ for all $1\leq i\leq n$.
	\end{thm} 
	Here by an indecomposable group we refer to a group $G$ which cannot be decomposed into a direct product of two nontrivial groups.
	
	The Krull-Schmidt Theorem obviously applies to finite groups, and in fact was first proved independently for finite groups by Wedderburn in 1909 (\cite{maclagan1909direct}). Observe that an infinite profinite group never satisfies the descending chain condition on its closed normal subgroups, and hence the Krull-Schmidt Theorem in this form cannot be applied for profinite groups.
	
	When considering profinite groups, we require all functions to be continuous. Hence, we say that a profinite group $G$ is indecomposable if whenever $G\cong A\times B$, an isomorphism in the category of profinite groups, then either $A$ or $B$ is a trivial group.
	
	A profinite group $G$ is said to be \emph{small} if for every $n \in \mathbb N$, $G$ has only finitely many open normal subgroups of index $n$. Examples of small groups are the finitely generated profinite groups (\cite[Proposition 2.5.1]{ribes2000profinite}). More generally the strongly complete profinite groups are small (\cite[Proposition 1]{smith2003subgroups}). It is worth mentioning that the class of small profinite groups is strictly bigger than the class of strongly complete profinite groups (\cite[Proposition 1]{segal2018remarks}).

	In this paper we prove a profinite version of the Krull-Schmidt Theorem for small profinite groups, and thus generalize the original statement of Wedderburn. Our argument is inspired by the proof that was given in the abstract case, however, we don't use the abstract, or even finite case, and the proof we give is self-contained.

	The main result of the paper is the following:
	\begin{thm}\label{Main Theorem}
		Let $G$ be a small profinite group. Then there exist countably many nontrivial indecomposable profinite groups $\{G_n\}_{ n\in I}$ with $I \subseteq \mathbb{N}$ such that $G\cong \prod_{n \in I} G_n$. Moreover, assume that 	$G\cong \prod_I H_i\cong \prod_J G_j$ are two decompositions of $G$ into direct products of nontrivial indecomposable profinite groups, then there is a bijective correspondence $h:I\to J$ such that for every $i\in I$, $H_i\cong G_{h(i)}$.   
	\end{thm}
	
	We will use the following characterisation of small profinite groups due to D. Segal. For a profinite group $G$ we denote by $\overline{G^m}$ the closure of the group $G^m= \langle g^m \ | \  \forall g \in G \rangle$ in $G$.
	
	\begin{thm}[{\cite[Theorem 1]{segal2018remarks}}]\label{segal characterization of small groups}
		A profinite group $G$ is small if and only if $\overline{G^m}$ is open in $G$ for every positive integer $m$.
	\end{thm}
	As every open subgroup of a profinite group $G$ contains $\overline{G^m}$ for some $m\in \mathbb{N}$, we conclude that when $G$ is small, the open subgroups $\overline{G^m}$ form a local basis of neighbourhoods of the identity.
	
	The subgroups $\overline{G^m}$ are extremely useful in our context since they respect direct products. More precisely, let $G\cong A\times B$, then $\overline{G^m}\cong \overline{A^m}\times \overline{B^m}$ for every $m\in \mathbb{N}$. Moreover, assume that $G\cong \prod_IG_i$ and $\overline{G^m}$ is open in $G$, then $\overline{G^m}\cong (\prod_{I\setminus J}\overline{G_i^m})\times (\prod_JG_i)$ where $J$ is the cofinite set of indices $i\in I$ such that $G_i\leq \overline{G^m}$. Here we identify each $G_i$ with its natural embedding into $G$. We leave the verification to the reader. 
	
	Our first step toward proving Theorem \ref{Main Theorem} is proving the \textit{cancellation} property in the category of small profinite groups.
	
	\begin{defn} Let $G$ be an abstract group. We say that $G$ \textit{can be cancelled from direct products}, or simply \textit{is cancellable}, in a category $\mathcal{C}$ of groups, if $G\times A\cong G\times B$ implies $A\cong B$ for every $A,B\in \mathcal{C}$ (and isomorphisms in the category $\mathcal{C}$). 
	\end{defn}
	
	The Krull-Schmidt Theorem \ref{kst} implies that every finite group is cancellable in the category of finite groups. Clearly, not every group - not even an abelian group - is cancellable within all (abelian) groups, as can be seen for example by taking $G=\Z^{\omega}, A=\Z, B=\Z\times \Z$. Much less trivial example is the group $\Z$. In \cite{hirshon1969cancellation} the author constructed two groups $A,B$ such that $A\ncong B$ but $\Z\times A\cong \Z\times B$. We remark that, contrary to the case of profinite groups that will be discussed in this paper, the groups $A,B$ are finitely generated.
	
	Cancellation of groups has been studied intensively, as can be seen for example in \cite{hirshon1970cancellation, hirshon1975cancellation, hirshon1977some, oger1983cancellation} and \cite{walker1956cancellation}. One of the most famous results in this field is the cancellable property of finite groups:
	\begin{thm}\cite{hirshon1969cancellation}\label{finite groups are cancellable}
		Let $G$ be a finite group and $A,B$ be arbitrary groups. Then $G\times A\cong G\times B$ implies $A\cong B$. 
	\end{thm}
	We use Theorem \ref{finite groups are cancellable} in order to prove a cancellation theorem for small profinite groups:
	
	\begin{prop}
		\label{f.g groups are cancellable}
		Every small profinite group is cancellable in the class of small profinite groups.
	\end{prop} 
	\begin{proof}
		Let $G,A,B$ be small profinite groups such that $G\times A\cong G\times B$. We wish to prove that $A\cong B$. As discussed above, taking the closed subgroup generated by all $n$-th powers in each side of the equation, one concludes that $(G\times A)/\overline{(G\times A)^n}\cong (G/\overline{G^n})\times (A/\overline{A^n})$ and $(G\times B)/\overline{(G\times B)^n}\cong (G/\overline{G^n})\times (B/\overline{B^n})$. We get that for every $n\in \mathbb{N}$, $(G/\overline{G^n})\times (A/\overline{A^n})\cong (G/\overline{G^n})\times (B/\overline{B^n})$. Since $G$ is a small profinite group by assumption, Theorem \ref{segal characterization of small groups} implies that  $G/\overline{G^n}$ is finite. By Theorem \ref{finite groups are cancellable} we conclude that for every $n$ $A/\overline{A^n}\cong B/\overline{B^n}$. Recall that for the same reason, the groups $A/\overline{A^n}, B/\overline{B^n}$ are finite too. Since $A,B$ are small, by \cite[Proposition 16.10.7]{fried2008field}) they are determined up to isomorphism by the set of their finite topological images. Clearly, every finite topological image of $A$ ($B$) is a quotient of $A/\overline{A^n}$ ($B/\overline{B^n}$) for some natural number $n$. Hence, $A\cong B$ and we are done.  
	\end{proof}
	
	The structure of the rest of the paper is as follows.
	In the next section we prove Theorem \ref{Main Theorem}. In section \ref{free products} we return to cancellation and prove that nontrivial free pro-$\mathcal C$ products are cancellable in suitable extension closed varieties $\mathcal C$ of finite groups. In the final section \ref{small groups} we prove that small profinite groups are precisely those profinite groups which are determined by the set of their topological finite images.
	
	\section{Proof of Theorem \ref{Main Theorem}}
	
	Let $G$ be a profinite group. As the following example shows $G$ cannot always be decomposed into a direct product of indecomposable groups. 
	\begin{exam}
		Recall the Pontryagin duality between abelian profinite groups and discrete abelian torsion groups (see \cite[Section 2.9]{ribes2000profinite}). Let $G$ be a profinite abelian group. If $G\cong \prod_I G_i$ is a direct product of nontrivial indecomposable groups, then $G^*\cong \oplus_IG_i^*$ is a direct sum of indecomposable torsion abelian groups (see \cite[Lemma 2.9.4]{ribes2000profinite}). Hence it is enough to find a torsion abelian group which cannot be decomposed as a direct sum of indecomposable groups and then take its Pontryagin dual. By \cite[Theorem 10]{kaplansky2018infinite} the only indecomposable torsion abelian groups are the cyclic groups of prime-power order, and the quasicyclic groups $C_{p^{\infty}}$. Now let $G$ be the abelian group that is generated by infinitely countable many variables $\{x_n\}_{n\in \mathbb{N}}$ subject to the relations $px_0=0; p^nx_n=x_0$. Clearly $G$ doesn't contain any copy of $C_{p^{\infty}}$. However, $G$ cannot be isomorphic to a direct product of finite cyclic groups as it has an element of infinite height. Hence, $G$ cannot be expressed as a direct sum of indecomposable groups. 
	\end{exam}
	When we consider small profinite groups, the situation is different. We restate the first part of Theorem \ref{Main Theorem} as follows.
	\begin{thm} \label{decomposition}
		Let $G$ be a small profinite group. Then there exists an indexing set $I$ and a set of indecomposable profinite groups $\{A_i\}_{i\in I}$  such that $G\cong \prod A_i$.
	\end{thm}
	\begin{proof}
		We look at the ordered set $(\mathcal{L},\preceq)$ of all decompositions $G\cong \prod_I A_i$ into nontrivial profinite groups. More precisely, an element in $\mathcal{L}$ is a tuple $(I,\{A_i\}_{i\in I},\psi)$ consisting of an indexing set, a set of nontrivial profinite groups, and an isomorphism $\psi:\prod_I A_i\to G$.  We say that $(I,\{A_i\}_{i\in I},\psi)\preceq(J,\{B_j\}_{j\in J},\varphi)$  if there exists a surjection $\varphi^{}_{\!JI}:J\to I$ such that for every $j\in J$ there is an embedding $\varphi_j:B_j\to  A_{\varphi^{}_{\!JI}(j)}$, which satisfies $\psi\circ\varphi_j|_{B_j}=\varphi|_{B_j}$. For more convenience we can replace every group by its image in $G$ and assume that all the groups are subgroups of $G$ and all the embeddings are inclusions. In particular we can omit all the homomorphisms in the definition.
		
		We want to show that there is a maximal element in $(\mathcal{L},\preceq)$, as such an element must consist of a set of indecomposable profinite groups. For that it is enough to show that every chain has an upper bound. First we shall notice that for every $(I,\{A_i\}_{i\in I})\preceq(J,\{B_j\}_{j\in J})$ and for every $i\in I$, $A_i$ is generated by the set of subgroups $\{B_j:\varphi_{JI}(j)=i\}$. Indeed, otherwise, there is some proper open subgroup, $U\leq A_j$ such that $A_j\leq U$ for all $\varphi_{JI}(j)=i$. But then the proper open subgroup $U\times\prod_{I\setminus \{i\}}A_i$ contains all the subgroups $B_j: j\in J$, in contrary to the fact that $\{B_j:j\in J\}$ generates $G$.
		
		Now let $(I_k,\{A_{i_k}\}_{i_k\in I_k})_{k\in S}$ be a chain in $(\mathcal{L},\preceq)$. We take $\{I,\varphi_{k}\}_{k\in S}$ to be the inverse limit of $\{I_k \ :\ k\in S\}$ together with maps $\phi_k: I \rightarrow I_k$. For each $i\in I$ we define $A_i:= \cap_{ k \in S} A_{\varphi_{k}(i)}$. We want to show that every open subgroup of $G$ contains all but finitely many of the $A_i$'s. Since $G$ is small it is enough to consider the open subgroups of the form $\overline{G^n}$. Denote $|\overline{G/G^n}|=m_n$. Let $(I_k,\{A_{i_k}\}_{i_k\in I_k})$ be an element in the chain. Then $G/\overline{G^n}\cong \prod_{I_k}A_{i_k}/\overline{A_{i_k}^n}$. Hence there are at most $m_n$ groups of the form $A_{i_k}$ which are not contained in $\overline{G^n}$. Letting $W_k:=\{  i_k \in I_k \  | \  A_{i_k} \not \subseteq \overline{G^n} \}$  and
		$W=\{  i \in I \  | \  A_{i} \not \subseteq \overline{G^n} \}$ we have $|W_k| \leq m_n$. Note that $\phi_k(W) \subseteq  W_k$ for each $k \in S$ and $i \in W$ is equivalent to $\cap_{k \in S} A_{\phi_k(i)} \not \subseteq  \overline{G^n}$ which is equivalent to $A_{\phi_k(i)} \not \subseteq  \overline{G^n}$, i.e. $\phi_k(i) \in W_k$ for every $k \in S$. Therefore $W$ equals the inverse limit of the sets $W_k$. As the size of $W_k$ is uniformly bounded by $m_n$ for each $k \in S$ we deduce that $|W| \leq m_n$ as required.
		
		Now we prove that the natural embeddings $A_i\to G$ induce an isomorphism $\prod_IA_i\to G$.
		
		Define a map $f: \prod_IA_i\to G$ by $f((a_i)_{i\in I})= \prod_Ia_i$. Observe that since the subgroups $A_i:i\in I$ commute and converge to 1, (i.e, every open subgroup of $G$ contains all but finitely many of them), the maps $f$ is well defined. Moreover, since the subgroups $A_i:i\in I$ commute, $f$ is indeed a homomorphism. Moreover, this map is clearly continuous.
		
		For showing that $f$ is an epimorphism, it is enough to show that the subgroups $A_i:i\in I$ generate $G$. This is equivalent to showing that there is no open subgroup of $G$ that contains all of them. Let $n\in \mathbb{N}$ be such that $G/\overline{G^n}$ is nontrivial. Let $(I_k,\{A_{i_k}\}_{i_k\in I_k})$ be an element in the chain. Since $\{A_{i_k}\}_{i_k\in I_k}$ generates $G$, there exists some $i_k\in I_k$ such that $A_{i_k}\nleq \overline{G^n}$. In other words the set $W_k$ defined earlier is nonempty. We showed that $W=\{  i \in I \  | \  A_{i} \not \subseteq \overline{G^n} \}$ is the inverse limit of the finite sets $W_k$ and hence $W \not = \emptyset$ as required.
		
		The only thing left to show is that $f$ is injective. Let $(a_i)_{i\in I}$ be such that $\prod_{i\in I} a_i=e$. We shall show that $a_i=e$ for all $i\in I$. It is enough to show that $a_i\in \overline{G^n}$ for every $n\in \mathbb{N}$ and $i\in I$. Let $n\in \mathbb{N}$. Let $J=\{ j_1,...,j_m\}$ be the finite set of indices $i\in I$ such that $A_i\nleq \overline{G^n}$. Then for all $i\in I\setminus J$, $a_i\in \overline{G^n}$. Let $(I_k,\{A_{i_k}\}_{i_k\in I_k})$ be an element in the chain such that  $\varphi_{I,I_k}(j_r)\ne \varphi_{I,I_k}(j_t)$ for all $r\ne t\in \{1,...,m\}$. Then $a_{j_t}\in A_{\varphi_{I,I_k}(j_1)}$ for all $1\leq t\leq m$.  The groups $A_{\varphi_{I,I_k}(j_1)},...,A_{\varphi_{I,I_k}(j_m)}$ generate their a direct product modulo $\overline{G^n}$. Hence $a_{j_1}\cdots a_{j_m}\leq \overline{G^n}$ implies that $a_{j_1},\ldots ,a_{j_m}\in \overline{G^n}$ as required.
		
		We proved that $G\cong_I \prod A_i$. However, some of the groups $A_i$ might be trivial. Replacing $I$ by $I'=\{i\in I:A_i\ne\{e\}\}$, clearly, $G\cong\prod_{i\in I'}  A_i$. The only thing left to show is that $I'$ projects over $I_k$ for every $k\in S$. In other words, we need to show that for every $i_k\in I_k$ there is some $i\in I, \varphi_{I,I_k}(i)=i_k$ such that $A_i\ne \{e\}$. Since $A_{i_k}\ne \{e\}$, it is enough to show that $A_{i_k}$ is generated by $\{A_i:\varphi_{I,I_k}(i)=i_k\}$. Since $A_{i_k}$ is generated by $\{A_{i_{k'}}:\varphi_{I_{K'},I_k}(i_{K'})=i_k\}$, for all $k\leq k'\in S$, and by assumption the subgroup they generate is their direct product, $A_{i_k}$ is naturally isomorphic to $\prod_{\varphi_{I_{k'},I_k}(i_{k'})=i_k}A_{i_{k'}}$ and the claim follows by the exact same proof for $G$. The proof of Theorem \ref{decomposition} is complete.
	\end{proof}
	In order to prove the uniqueness part of Theorem \ref{Main Theorem} we first need some background on \textit{normal endomorphisms}.
	\begin{defn}
		Let $G$ be a profinite group and $f:G\to G$ an endomorphism. We say that $f$ is \textit{normal} if for every $a,b\in G$, $af(b)a^{-1}=f(aba^{-1})$. 
	\end{defn}
	\begin{rem}
		Let $G$ be a profinite group and $f:G\to G$ a normal endomorphism. Assume that $H\leq G$ is a closed normal subgroup $f$-invariant subgroup of $G$ such that $f$ induces a homomorphism $\bar{f}:G/H\to G/H$. Then $\bar{f}$ is also normal.
	\end{rem}
	\begin{defn}
		Let $\psi,\varphi$ be two endomorphisms of $G$. Define \emph{the sum} \[ \varphi \odot \psi: \  G \rightarrow G \] of $\varphi$ and $\psi$ to be the function defined by $(\varphi \odot\psi)(a)=\varphi(a)\psi(a)$ for all $a \in G$.
	\end{defn}
	As usual we reserve the notation $\varphi \psi$ for the composition $\varphi \circ \psi$.
	
	\begin{lem}\label{properties of normal}
		Let $G$ be a profinite group and $\psi,\varphi$ be normal endomorphisms on $G$. The reader may easily verify the following properties: 
		\begin{enumerate}
			\item If $\varphi \odot \psi$ is an endomorphism of $G$ then it is a normal endomorphism as well.
			\item The composition $\varphi\circ\psi$ is a normal endomorphism of $G$.
			\item If $\psi$ is an automorphism then $\psi^{-1}$ is a normal auomorphism as well.
		\end{enumerate}
	\end{lem}
	A fundamental example of a normal automorphism is the following: 
	\begin{lem}\label{every sum is normal}
		Let $G=H_1\times \cdots H_m$ and denote by $\pi_j:G\to H_j$ and $i_j:H_j\to G$ the natural projections and inclusion maps. Then the sum of any $k$ distinct $i_j\pi_j$ is a normal endomorphism. Observe that the sum of all of them is the identity.
	\end{lem}
	\begin{defn}
		Let $f:G\to G$ be an endomorphism. We say that $f$ is \textit{nilpotent} if for some $n\in \mathbb{N}$, $f^n(G)=\{e\}$.
	\end{defn}
	Now let $G$ be an indecomposable small profinite group, and let $f:G\to G$ be a normal endomorphism. Then for each $n \in \mathbb N$, $f$ naturally induces an endomorphism $f_n:G/\overline{G^n}\to G/\overline{G^n}$ of the finite group $G/\overline{G^n}$.
	\begin{lem}\label{automorphism or nilpotent}
		Let $G$ be an indecomposable profinite group and $f:G\to G$ a normal endomorphism. Then either $f$ is an automorphism or for every $n$ $f_n$ is nilpotent.
	\end{lem}
	In order to prove the above lemma, we first need to recall a lemma on finite groups.
	\begin{lem}\label{Standard exercise}
		Let $G$ be a finite group and $f:G\to G$ a normal endomorphism. The following are equivalent:
		\begin{enumerate}
			\item $G= \ker f^n\oplus \operatorname{Im}f^n$.
			\item $\ker f^n=\ker f^k$ and $\operatorname{Im}f^n=\operatorname{Im}f^k$ for every $k\geq n$.
		\end{enumerate}
		\begin{proof}
			$(1)\Rightarrow(2):$ It is enough to prove that $\ker f^n=\ker f^{2n}$ and $\operatorname{Im}f^n=\operatorname{Im}f^{2n}$, since then we may replace $f^n$ by $f^{2^mn}$ for all natural number $m$ and conclude by induction that $\ker{f^n}=\ker f^{2^mn}$ and $\operatorname{Im}f^n=\operatorname{Im}f^{2^mn}$ for all natural number $m$. The rest of equalities are followed by the inclusions $\ker f^{2^mn}\subseteq \ker f^k\subseteq f^{2^{m+1}n}$ and $\operatorname{Im} f^{2^{m+1}n}\subseteq \operatorname{Im} f^k\subseteq \operatorname{Im} f^{2^{m}n}$ for a suitable $m$.
			
			Let $x\in \ker f^{2n}$. Express $x=yf^n(z)$ for some $y\in \ker f^n$. Then $f^n(x)=f^n(y)f^{2n}(z)=f^{2n}(z)$. Hence $f^n(x)=f^{2n}(z)\in \ker f^n\cap \operatorname{Im}f^n=\{e\}$. We conclude that  $x\in \ker f^n$. We get that $\ker f^n=\ker f^{2n}$. Since $\operatorname{Im}f^{2n}\subseteq \operatorname{Im} f^n$ and $|\operatorname{Im}f^{2n}|=\dfrac{|G|}{|\ker f^{2n}|}=\dfrac{|G|}{|\ker f^{n}|}=|\operatorname{Im}f^{n}|$ then $\operatorname{Im}f^{2n}=\operatorname{Im}f^{n}$ as required.

			$(2)\Rightarrow(1):$ We shall show that $\ker f^n\cap \operatorname{Im} f^n=\{e\}$. Indeed, let $x\in \ker f^n\cap \operatorname{Im} f^n$. Then $x=f^n(y)$ for some $y\in G$. Observe that $f^{2n}(y)=f^n(x)=e$. Hence $y\in \ker f^{2n}=\ker f^n$ which implies that $x=f^n(y)=e$. Since $\ker f^n$ and $\operatorname{Im} f^n$ are both normal subgroups we get that the subgroup they generate is their direct product. Since $|G|=|\ker f^n|\cdot|\operatorname{Im}{f^n}|=|\ker f^n\oplus \operatorname{Im}{f^n}|$ we conclude that $G=\ker f^n\oplus \operatorname{Im}{f^n}$ as required.
		\end{proof}
	\end{lem}
	\begin{proof}[Proof of Lemma \ref{automorphism or nilpotent}]
		Since $G/\overline{G^n}$ are finite groups, for every $n$ there exists some minimal number $m_n\in \mathbb{N}$ such that $\ker(f_n^{m_n})=\ker(f_n^k)$ and $\operatorname{Im}(f_n^{m_n})=\operatorname{Im}(f_n^k)$ for every $k\geq m_n$. Since $f_n$ is normal, Lemma \ref{Standard exercise} shows that \[ G/\overline{G^n}= \ker(f_n^{m_n})\oplus \operatorname{Im}(f_n^{m_n}).\] 
		
		Observe that if $n_1|n_2$ then $m_{n_1}\leq m_{n_2}$. Indeed, we have 
		
		\begin{equation} \label{e1} G/\overline{G^{n_2}}= \ker(f_{n_2}^{m_{n_2}})\oplus \operatorname{Im}(f_{n_2}^{m_{n_2}}) \end{equation} and hence $\ker(f_{n_2})^{m_{n_2}}\overline{G^{n_1}}$ and $\operatorname{Im}(f_{n_2})^{m_{n_2}}\overline{G^{n_1}}$ generate $G/\overline{G^{n_1}}$. The maps $\{f_n\}_{n\in \mathbb{N}}$ are compatible with the natural projections $\pi_{n_1,n_2}: G/\overline{G^{n_2}} \to G/\overline{G^{n_1}}$ and therefore \begin{equation} \label{e2} \ker(f_{n_2}^{m_{n_2}})\overline{G^{n_1}}\leq \ker(f_{n_1}^{m_{n_2}}) \textrm{ and }  \operatorname{Im}(f_{n_2}^{m_{n_2}})\overline{G^{n_1}}\leq \operatorname{Im}(f_{n_1}^{m_{n_2}}).\end{equation} 
		
		Hence $\ker(f_{n_1}^{m_{n_2}})$ and $\operatorname{Im}(f_{n_1}^{m_{n_2}})$ generate $G/\overline{G^{n_1}}$. Since both groups are normal, and obviously $$|\ker(f_{n_1}^{m_{n_2}})|\cdot |\operatorname{Im}(f_{n_1}^{m_{n_2}})|=|G/\overline{G^{n_1}}|,$$ it follows that \begin{equation} \label{e3}G/\overline{G^{n_1}}= \ker(f_{n_1}^{m_{n_2}})\oplus \operatorname{Im}(f_{n_1}^{m_{n_2}}).\end{equation}
		
		Lemma \ref{Standard exercise} now implies that $\ker(f_{n_1}^{m_{n_2}})=\ker(f_{n_1}^k)$ and	 $\operatorname{Im}(f_{n_1}^{m_{n_2}})=\operatorname{Im}(f_{n_1}^k)$ 
		for every $k\geq m_{n_2}$. In particular $m_{n_1}\leq m_{n_2}$ and $\ker(f_{n_1}^{m_{n_1}})=\ker(f_{n_1}^{m_{n_2}})$,
		$\operatorname{Im}(f_{n_1}^{m_{n_1}})=\operatorname{Im}(f_{n_1}^{m_{n_2}})$. In addition (\ref{e1}), (\ref{e2}) and (\ref{e3}) imply $$\ker(f_{n_2}^{m_{n_2}})\overline{G^{n_1}}= \ker(f_{n_1}^{m_{n_2}}) =\ker(f_{n_1}^{m_{n_1}})$$  and $$\operatorname{Im}(f_{n_2}^{m_{n_2}})\overline{G^{n_1}}= \operatorname{Im}(f_{n_1}^{m_{n_2}})=\operatorname{Im}(f_{n_1}^{m_{n_1}}).$$
		We conclude that for every $n_1|n_2$ the natural projection $\pi_{n_1,n_2}$ sends $\ker(f_{n_2}^{m_{n_2}})$ onto $\ker(f_{n_1}^{m_{n_1}})$ and $\operatorname{Im}(f_{n_2}^{m_{n_2}})$ onto $\operatorname(f_{n_1}^{m_{n_1}})$.
		
		Let $H_1= {\underleftarrow{\lim}}_n \ker(f_n^{m_n})$ and $H_2={\underleftarrow{\lim}}_n \operatorname{Im}(f_n^{m_n})$, then $G= H_1\oplus H_2$. Recall that $G$ is indecomposable and hence $H_2$ is either $G$ or $\{e\}$. 
		
		Suppose that $H_2=G$. Then for every $n$ $\operatorname{Im}f_n^{m_n}=G/\overline{G^n}$. This implies that $f_n^{m_n}$ and hence $f_n$ is surjective for every $n$ and hence an automorphism. We conclude that $f$ is an automorphism.
		
		Suppose $H_2=\{e\}$. Then for every $n$ $\operatorname{Im}f_n^{m_n}=\{e\}$ and thus $f_n$ is nilpotent. 
	\end{proof}
	\begin{lem} \label{two}
		Let $G$ be a small indecomposable profinite group and $\varphi,\psi:G\to G$ normal endomorphisms such that $\varphi \odot \psi$ is an endomorphism. Assume that $\varphi \odot \psi$ is an automorphism. Then either $\varphi$ or $\psi$ is an automorphism. 
	\end{lem}
	\begin{proof}
		Let $\gamma:G\to G$ be the inverse of $\varphi \odot \psi$. Then by Lemma \ref{properties of normal} $\gamma$ is normal too, as well as $\Lambda=\varphi \gamma$ and $\mu=\psi \gamma$. Observe that $\Lambda \odot \mu=(\varphi \odot \psi)\gamma=Id_G$. For every $n\in \mathbb{N}$ we denote by $\varphi_n,\psi_n,\gamma_n,\Lambda_n,\mu_n$ the induced endomorphisms $G/\overline{G^n}\to G/\overline{G^n}$. They are all normal. Assume by contradiction that $\varphi$ and $\psi$ are not automorphisms. Then $\Lambda$ and $\mu$ are not automorphisms either. By Lemma \ref{automorphism or nilpotent} we conclude that for every $n$, $\Lambda_n$ and $\mu_n$ are nilpotent. 
		
		Since $Id_G=(\varphi \odot \psi)\gamma=\Lambda \odot \mu$ one has $x=\Lambda(x)\mu(x)$ for every $x\in G$. Taking inverses we get $x^{-1}=\mu(x^{-1})\Lambda(x^{-1})$ for every $x\in G$. Hence $\mu \odot \Lambda=\Lambda \odot \mu$. In particular, $\mu_n \odot \Lambda_n=\Lambda_n  \odot \mu_n$ for every $n$. Moreover, since $\Lambda_n \odot \mu_n=Id_{G/\overline{G^n}}$ then $\Lambda_n(\Lambda_n \odot \mu_n)=(\Lambda_n \odot \mu_n)\Lambda_n$ which implies that $\Lambda_n\mu_n=\mu_n\Lambda_n$. One concludes that $\Lambda_n \odot \mu_n$ is nilpotent and in particular not onto whenever $G/\overline{G^n}$ is nontrivial. Since there must be some $n$ for which $G/\overline{G^n}$ is nontrivial, this is contradiction to $\Lambda \odot \mu=Id_G$.
	\end{proof}
	By induction Lemma \ref{two} generalizes to the following.
	\begin{cor}\label{sum of many normals implies an automorphism}
		Let $G$ be an indecomposable small profinite group. If $\psi_1, \ldots , \psi_n$ are normal endomorphisms whose sum $\psi_1 \odot \cdots \odot \psi_n$ is an automorphism and every partial sum of them is an endomorphism, then at least one of them is an automorphism.
	\end{cor}
	Now we are ready to prove Theorem \ref{Main Theorem}.
	
	\begin{proof}
		First observe that $I$ (and $J$) must be countable. Indeed, if $G$ is small then it has countably many open subgroups. Assume by contradiction that $I$ is uncountable. For every $i\in I$ choose a proper open subgroup $U_i$ of $G_i$. Then the set $\{U_i\times \prod_{i\ne j\in I}H_j\}_{i\in I}$ is a set of uncountably many open subgroups of $G$, a contradiction. inserting several trival factors to $\{H_i\}_I$ and $\{G_j\}_J$ if necessary we may assume that $I=J= \mathbb N$. 
		
		It is enough to show that for every $i\in I$ there is $j\in J$ such that $H_j \cong G_i$. Let us call it \textit{Property $P$}. Indeed, assume that property $P$ holds, then by symmetry, for every $j\in J$ there exists $i\in I$ such that $G_i\cong H_j$. Now let $K$ be an indecomposable small profinite group. We need to how that the cardinality of the set $\{i \in I \ | \ H_i \cong K\}$ is the same as $\{j \in J \ | \ G_j \cong K\}$.
		
		Case 1: For every $i\in I$, $H_i\ncong K$. Then by Property $P$,  for every $j\in J$, $G_j\ncong K$.
		
		Case 2: There are finitely many indices $i\in I$ such that $H_i\cong K$. Without loss of generality assume that $H_1\cong \cdots \cong H_n\cong K$ and for every $i\geq n$ $H_i\ncong K$. We prove the claim by induction on $n$. We have $G= \prod_{i=1}^\infty H_i$. By property $P$ there exists some $j\in J$ such that $H_1\cong G_j$and without loss of generality we can assume that $j=1$. Put $G'=\prod_{i\geq 2}H_i $, $G''=\prod_{j\geq 2}G_j$.  Then $H_1\times G'\cong G_1\times G''$. Observe that the groups $H_1,G_1,G',G''$ are all small as quotients of a small profinite group. Since $H_1\cong G_1$, by cancellation in small profinite groups we conclude that $G'\cong G''$. If $n=1$ then $G'$ has no components isomorphic to $K$, and hence by the first case so does $G''$ and we are done. Otherwise, we apply the induction assumption.
		
		Case 3: There are infinitely many copies of $K$ in $\{H_i \}_{i \in I}$ then by successive cancellation as in case  2 we can conclude that  same is true for $\{G_j\}_{j \in J}$ and we are done.
		
		Now we prove property $P$. Without loss of generality we can assume $i=1$.
		
		Let us denote by $\psi_1:H_1\to G$, $\psi_k':G_k\to G$ the natural embeddings, and by $\pi_1:G\to H_1$, $\pi_k':G\to G_k$ the natural projections. Observe that $f_k:=\pi_1  \psi_k' \pi_k'  \psi_1$ are normal endomorphisms of $H_1$ such that for all $x \in H_1$ $f_k(x) \rightarrow e_{H_1}$ as $k \rightarrow \infty$. It follows that the infinite sum $\bigodot_{k \in J} f_k$ is converging and equals $Id_{H_1}$, in particular it is an automorphism of $H_1$. 
		
		Let $n\in \mathbb{N}$ be such that $H_1/\overline{H_1^n}$ is nontrivial. Since $\overline{H_1^n}$ is open in $H_1$ there is a finite set $T\subseteq J$ such that for every $k\in J\setminus T$, the induced endomorphism  $f_k$ is trivial on $H_1/\overline{H_1^n}$. Then $\bigodot_{ k \in T} f_k$ is the identity map, and in particular an automorphism, on $H_1/\overline{H_1^n}$. In particular, it is not nilpotent. 
		By Lemma \ref{automorphism or nilpotent} we conclude that $\bigodot_{k \in T} f_k$ is an automorphism on $H_1$. In addition, $\bigodot_{k \in S} f_k$ is a normal endomorphism on $H_1$ for every $S\subseteq T$. By Corollary \ref{sum of many normals implies an automorphism} one of the summands $f_k$ must be an automorphism of $H_1$. Assume that $f_j=\pi_1 \psi_j' \pi_j'  \psi_1 \in \mathrm{Aut}(H_1)$.
		
		The only thing left to show is that $\pi_j' \psi_1$ is an isomorphism. Let $\gamma=f_j^{-1}$. Then $(\gamma \pi_1 \psi_j' )(\pi_j'\psi_1)=Id_{H_1}$. We shall show that $(\pi_j' \psi_1)(\gamma \pi_1 \psi_j' )=Id_{G_j}$. Put $\sigma =(\pi_j'\psi_1)(\gamma \pi_1 \psi_j' ) \in \mathrm{End}(G_j)$. Then $\sigma^2=\sigma$. One checks that $\sigma$ is normal. Hence, since $\operatorname{Im}\sigma$ is normal subgroup of $G_j$, Lemma \ref{Standard exercise} gives $G_j\cong \ker \sigma\oplus \operatorname{Im}\sigma$. Since $G_j$ is indecomposable, either $\operatorname{Im}\sigma=G_j$ or $\ker\sigma=G_j$. If $\operatorname{Im}\sigma=G_j$ then $\sigma$ is an epimorphism and thus an automorphism. Hence since $\sigma^2=\sigma$, $\sigma=Id_{G_1}$. Otherwise $\sigma$ is the trivial endomorphism. But this is impossible since $$Id_{H_1}=Id_{H_1}^2= (\gamma \pi_1 \psi_j'\pi_j'\psi_1)^2=(\gamma \pi_1 i_1' )\sigma (\pi_1'i_1)$$ and $H_1$ is nontrivial.
		
		This concludes the proof of Theorem \ref{Main Theorem}.
	\end{proof}
	We do not know if Theorem \ref{Main Theorem} applies more generally and ask:
	\begin{ques} \label{uniqueness in the general case}
		Let $G$ be a (non-small) profinite group. Assume that $G\cong\prod_JG_i\cong\prod_JH_j$ such that for all $i\in I,j\in J$ $G_i,H_j$ are nontrivial indecomposable profinite groups. Is there a bijective function $h:I\to J$ such that $G_i\cong H_{h(i)}$ for all $i\in I$? 
	\end{ques}
	\begin{rem}
		It is worth mentioning that for profinite abelian groups Question \ref{uniqueness in the general case} has a positive answer. Indeed every profinite abelian group $G$ is a product of its Sylow subgroups and hence it suffices to consider abelian pro-$p$ group. By \cite[Theorem 10]{kaplansky2018infinite} the only indecomposable abelian pro-$p$ groups are $\Z_p$ and the finite cyclic groups $C_{p^n}$. Hence we may assume that $G= (\Z_p)^{\alpha} \times \prod_{k=1}^\infty (C_{p^k})^{\beta_k}$ for some cardinals $\alpha$ and $\beta_k$. But then $G_0:=\prod_{k=1}^\infty (C_{p^k})^{\beta_k}$ is uniquely determined as the closure of the torsion elements of $G$ and in particular $\alpha$ is the rank of the free abelian pro-$p$ group $G/G_0$ while $\beta_k$ is the rank of the kernel of the map $f: G_0^{p^{k-1}}/G_0^{p^{k}} \rightarrow G_0^{p^{k}}/G_0^{p^{k+1}}$ induced by the map $g \mapsto g^p$.
	\end{rem}
	\section{Cancellation of free products} \label{free products}
	Cancellation of direct factors in small profinite groups played an important role in our proof of Theorem \ref{Main Theorem}. In this section we prove some more results regarding cancellation. We focus mainly on cancellation of some free pro-$\mathcal{C}$ products, where $\mathcal C$ is an extension closed variety of finite groups. Recall the definition of a free pro-$\mathcal{C}$ product over a sheaf of pro-$\mathcal{C}$ groups from \cite[Chapter 5]{ribes2017profinite}. First we need some results on the centralizers in a free pro-$p$ product.
	\begin{thm}\label{centalizers in free product}
		Let $F$ be a pro-$p$ group which is a non trivial free pro-$p$ product. Assume that $F=G_1\coprod_{p} G_2$ and let $x\in F$ be an element which is not in the conjugate class of any $G_i$. Then $C_F(x)$ is infinite pro-cyclic.
	\end{thm}
	\begin{proof}
		Let $y_1,...,y_n\in C_F(x)$. Consider $H=\overline{\langle x,y_1,...,y_n\rangle }$. Since $H$ is finitely generated, by \cite[Theorem 9.6.2]{ribes2017profinite}, there is complete set of representatives $\{h_{j,i} \}_{j \in I_i}$ for the double cosets $D_i= H \backslash F /G_i$, such that  $H=[\coprod_{D_1 \cup D_2}(H_p\cap h_{j,i}G_ih_{j,i}^{-1} )]\coprod F_p$, where $F_p$ is a free pro-$p$ group. Now assume that for some pair $j,i$ we have $H\cap h_{j,i}G_ih_{j,i}^{-1}\ne \{e\}$ and let $g\in H\cap h_{j,i}G_ih_{j,i}^{-1}$. In particular $g\in h_{j,i}G_ih_{j,i}^{-1}$ which is a free factor of $F$. Hence by \cite[Theorem 9.1.12]{ribes2017profinite}, $ C_F(g)=C_{h_{j,i}G_ih_{j,i}^{-1}}(g)\leq h_{j,i}G_ih_{j,i}^{-1}$. In particular, $x\in C_F(g)\leq h_{j,i}G_ih_{j,i}^{-1}$, a contradiction. Thus $H=F_p$ is a free pro-$p$ group. Since $H$ has a nontrivial center by \cite[Corollary 8.7.3]{ribes2000profinite} $H \cong\Z_p$. 
		
		Now let $A$ be a finite continuous quotient of $C_F(x)$. There are $y_1,...,y_n\in C_F(x)$ such that  $H=\overline{\langle x,y_1,...,y_n\rangle }$ projects over $A$. Thus $A$ is cyclic. So   $C_F(x)$ is pro-cyclic. It is infinite since it contains an infinite subgroup.
	\end{proof}
	
	\begin{thm}\label{free pro-p product can be canceled}
		A nontrivial free pro-$p$ product can be cancelled in direct product. 
	\end{thm}
	\begin{proof}
		The proof is based on the proof of \cite[Theorem 6]{hirshon1977some}.
		
		Let $G$ be a profinite group which can be decomposed as $A\times F$ and $B\times F'$ for $F\cong F'$ a nontrivial free pro-$p$ product. Then $F'\cong G/B= AB/B\cdot FB/B$, a product of two element-wise commuting subgroups. Denote $A_1=AB/B$ and $F_1=FB/B$ for more convenience, and assume first that $A_1,F_1\ne \{e\}$. Decompose $F'$ to a nontrivial free pro-$p$ product as $F'\cong K_1\coprod K_2$. If $A_1$ intersects nontrivialy with some conjugate of a factor, say $A_1\cap gK_ig^{-1}\ne \{e\}$ for some $i=1\lor 2$ and $g\in F'$, then we can take some $e\ne x\in A_1\cap gK_ig^{-1}$. Recalling again that $A_1$ and $F_1$ are element-wise commuting, by \cite[Theorem 9.1.12]{ribes2000profinite} $F_1\leq C_{F'}(x)=C_{gK_ig^{-1}}\leq gK_ig^{-1}$. Moreover, choosing some nontrivial element $y\in F_1$ we conclude by the same way that $A_1\leq gK_ig^{-1}$. Hence we get that $F'\leq gK_ig^{-1}$, which implies that $F'=g^{-1}F'g\leq K_i$, contradicting the nontriviality of the decomposition. By symmetry we conclude that $A_1$ and $F_1$ intersect trivially with every conjugate of every factor. Let $e\ne g\in F_1$, then by \ref{centalizers in free product} $A_1\leq C_{F'}(g)$ which is a pro-cyclic group and in particular abelian. Thus $A_1$ is abelian. Now let $e\ne a\in A_1$. then by the commutativity of $F_1$ with $A_1$, $F_1,A_1\leq C_{F'}(a)$, which is a pro-cyclic group. We conclude that $F'$ is pro-cyclic, contradicting the nontriviality of the decomposition as a free pro-$p$ product. Hence either $A_1$ or $F_1$ are trivial.
		
		First case: $A_1=\{e\}$.  It means that $A\subseteq B$. Factoring out by $A$ we obtain $F\cong B/A\times F'$. As we already saw, a nontrivial free pro-$p$ product cannot be generated by two nontrivial commuting subgroups, hence $B/A=\{e\}$ and $A=B$.
		
		Second case: $F_1=\{e\}$. It means that $F\subseteq B$. We get that $B=(A\cap B)\times F$. So $A\times F=B\times F'=(A\cap B)\times F\times F'$. Dividing by $F$ we get that $A=(A\cap B)\times F'$. Since $F\cong F'$, we conclude that $A\cong B$.  
	\end{proof}
	
	Since a free pro-$p$ group is a free pro-$p$ product of copies of $\Z_p$ (see \cite[Example 5.1.3]{ribes2017profinite}), we obtain
	\begin{cor}
		A free pro-$p$ group of rank greater than 1 is cancellable.
	\end{cor}
	We can generalize the above corollary to free pro-$\mathcal{C}$ groups for some other varieties. First we need the following proposition: 
	\begin{prop}\label{free product generated by commuting subgroups}
		Let $\pi$ be a set of primes and $\mathcal{C}$ be an extension closed variety of finite groups which contains $C_p$ for every $p\in \pi$. Let $F$ be a pro-$\pi$ group which is a non trivial free pro-$\mathcal{C}$ product. If $F$ is generated by two commuting subgroups $A,B$, then either every $p$-Sylow subgroup is infinite pro-cyclic or $F\cong A\times B$ and $(o(A),o(B))=1$. 
	\end{prop}
	\begin{proof}
		Since $F$ is isomorphic to an image of $A\times B$, then for every $p$, there exists a $p$-Sylow subgroup of $F$ of the form $A_pB_p$ for $A_p,B_p$ which are $p$-Sylow subgroups of $A,B$ correspondingly. Consider the decomposition $F=K_1\coprod_{\mathcal{C}}K_2$. Fix some prime $p\in \pi$ and assume by contradiction that $A_p,B_p\ne \{e\}$. Let $g\in A_p$. First observe that $g$ is not in any conjugate of $K_i$, for otherwise, by  \cite[Theorem 9.1.12]{ribes2000profinite}, $B\leq C_F(g)=C_{x^{-1}K_ix}(g)\leq x^{-1}K_ix$. But then for every $b\in B$, $A\leq C_F(b)=C_{x^{-1}K_ix}(b)\leq x^{-1}K_ix$ and we get that the decomposition of $F$ is trivial.
		
		Now let $P$ be a $p$-Sylow subgroup of $C_F(g)$ containing $\overline{\langle g\rangle }B_p$- observe that this is indeed a pro-$p$ group in the centralizer. 
		Let $y_1,...,y_n\in P$ and let $H=\overline{\langle g,y_1,...,y_n\rangle }$. Since $H$ is finitely generated, by \cite[Theorem 9.6.2]{ribes2017profinite}, there is complete set of representatives $\{h_{j,i} \}_{j \in I_i}$ for the double cosets $D_i= H \backslash F /K_i$, such that  $H=[\coprod_{D_1 \cup D_2}(H_p\cap h_{j,i}K_ih_{j,i}^{-1} )]\coprod F_p$, where $F_p$ is a free pro-$p$ group. Now assume that for some pair $j,i$ we have $H\cap h_{j,i}K_ih_{j,i}^{-1}\ne \{e\}$ and let $x \in H\cap h_{j,i}K_ih_{j,i}^{-1}$. In particular $x\in h_{j,i}K_ih_{j,i}^{-1}$ which is a free factor of $F$. Hence by \cite[Theorem 9.1.12]{ribes2017profinite}, $ g \in C_F(x)=C_{h_{j,i}K_ih_{j,i}^{-1}}(g)\leq h_{j,i}K_ih_{j,i}^{-1}$, a contradiction to the choice of $g$. Thus $H=F_p$ is a free pro-$p$ group. Since $H$ has a nontrivial center by \cite[Corollary 8.7.3]{ribes2000profinite} $H \cong\Z_p$.

		Now let $T$ be a finite continuous quotient of $P$. There are $y_1,...,y_n$ such that  $H=\overline{\langle g,y_1,...,y_n\rangle }$ projects over $T$. Thus $T$ is cyclic. So $P$ is pro-cyclic. As $B_p\leq P$, $B_p$ is commutative. Now let $g'\in B_p$. By the same proof every $p$- Sylow subgroup of $C_{F}(g')$ is isomorphic to $\mathbb{Z}_p$. But $A_pB_p$ is a $p$-Sylow subgroup of $F$ and hence of $C_F(g').$ In particular, every $p$-Sylow subgroup of $F$ is infinite procyclic.  
		
		Otherwise, $(o(A),o(B))=1$. In particular, $A\cap B=\{e\}$ and $F\cong A\times B$.
	\end{proof}
	\begin{thm}
		Let $\mathcal{C}$ be an extension-closed variety which is not contained in the variety of solvable groups (for example, the variety of all finite groups) and denote by $\pi(\mathcal{C})$ the set of primes involved in $\mathcal{C}$. Assume further that for every $p\in \pi(\mathcal{C})$, $C_p\in \mathcal{C}$. Then a free pro-$\mathcal{C}$ group of rank greater than one can be cancelled in direct products.
	\end{thm}
	\begin{proof}
		Let $F$ be a free pro-$\mathcal{C}$ group of rank greater than 1. By [Example 5.1.3, profinite graphs], $F$ is isomorphic to a nontrivial free pro-$\mathcal{C}$ product. Assume that $A\times F\cong A_1\times F_1$ and let $A'=A_1A/A, F'=F_1A/A$. By Proposition \ref{free product generated by commuting subgroups}, if $A',F'\ne \{e\}$, then either $F$ has pro-cyclic $p$-Sylow subgroups, or $F$ decomposes as a nontrivial free product. If $F$ has pro-cyclic $p$-Sylow subgroups then the same holds for every finite quotient of $F$. Recalling the result of Hall \cite[Theorem 9.4.3]{hall2018theory} which implies that a finite group all whose $p$-Sylow subgroups are cyclic is solvable, we get that all finite quotients of $F$ are solvable. However, by assumption, $\mathcal{C}$ contains non-solvable groups. Since $\mathcal{C}$ is a variety it is closed under taking normal subgroups and quotients, and hence must contain a non-abelian finite simple group. Recall that every finite simple group can be generated by two elements (\cite{steinberg1962generators}). As $F$ is a free pro-$\mathcal{C}$ group, it projects onto all groups in $\mathcal{C}$ with rank less or equal then the rank of $F$. In particular, $F$ admits non-abelian finite simple quotient, which is hence non-solvable. Hence, this case cannot happen. Hence, $F$ decomposes as a nontrivial free product. However, by \cite[Proposition 8.7.7]{ribes2000profinite} for any NE-formation $\mathcal{C}$, a free pro-$\mathcal{C}$ group of rank greater than 1 cannot be decomposed as a nontrivial direct product. We conclude that either $A'$ or $F'$ must be trivial. The rest of the proof is identical to the proof of Theorem \ref{free pro-p product can be canceled}.
	\end{proof}

	The case of a free pro-$p$ group of rank 1 remains open. We present two special cases in which $\Z_p$ can be canceled.
	\begin{lem}
		Assume $\mathbb{Z}_p\times \mathbb{Z}_p\times H\cong \mathbb{Z}_p\times K$. Then $K\cong \mathbb{Z}_p\times H$.
	\end{lem}	
	\begin{proof}
		The proof is similar to the proof of \cite[Lemma 1]{hirshon1977some}. Write $G=\overline{\langle w\rangle }\times \overline{\langle y\rangle }\times  H=\overline{\langle g\rangle }\times K$. 
		
		First case: $g\notin  \overline{\langle y\rangle }\times H, \overline{\langle w\rangle }\times  H$. Write $g=y^aw^bh$ for $0\ne a,b\in \mathbb{Z}_p$ and $h\in H$. Put $a=p^ta', b=p^sb'$ for $t,s\in \{0\}\cup \mathbb{N}$ and $a',b'\in (\Z_p)^{\times}$. Assume without loss of generality that $s\leq t$.  Put $w'= y^{p^{t-s}a'}w^{b'}$. Then $g=(w')^{p^s}h$.  Observe that since $b'\in (\Z_p)^{\times}$, $\overline{\langle y,w\rangle }=\overline{\langle y,w'\rangle }$. Thus we can replace $w$ by $w'$ and get that $g\in \overline{\langle w'\rangle }\times  H$.
		
		Second case: $g\in  \overline{\langle w\rangle }\times H$. Hence $\overline{\langle w\rangle }\times H=\overline{\langle g\rangle }\times K_1$ for $K_1=K\cap (\overline{\langle w\rangle }\times H)$. Thus, $G=\overline{\langle y\rangle }\times \overline{\langle g\rangle }\times K_1=\overline{\langle g\rangle }\times K$. Dividing both sides by $\overline{\langle g\rangle }$ we conclude that $\overline{\langle y\rangle }\times K_1\cong K$. But $\Z_p\times H\cong \overline{\langle w\rangle }\times H=\overline{\langle g\rangle }\times K_1\cong \overline{\langle y\rangle }\times K_1\cong K$ and we are done.
	\end{proof}
	\begin{lem}
		Let $G_1,G_2$ be profinite abelian groups such that $\Z_p\times G_1\cong \Z_p\times G_2$. Then $G_1\cong G_2$.
	\end{lem}
	\begin{proof}
		Let $G=H_1\times G_1=H_2\times G_2$ be a profinite group having two inner decompositions as a direct product, such that $H_1 \cong H_2\cong \Z_p$. Let $H=G_1\cap G_2$.
		
		First case: Without loss of generality $H=G_1$. Then $\Z_p\cong G/G_1$ naturally projects over $G/G_2\cong \Z_p$. However, every projection of $\Z_p$ over itself is an isomorphism. Hence $G_1=G=2$.
		
		Second case: $H$ is a proper subgroup of $G_1$ and $G_2$. Observe that $G_i/H\cong \Z_p$ for $i=1,2$. Indeed, $G_i/H\cong G_iG_{3-i}/G_{3-i}\leq G/G_{3-i}\cong \Z_p$. As $G_i/H$ is nontrivial we conclude that $G_i/H\cong \Z_p$. Since $\Z_p$ is free, the projections $G_i\to G_i/H\cong \Z_p$ split. As the groups $G_i$ are abelian, we conclude that that $G_i\cong H\times \Z_p$ for $i=1,2$. In particular, $G_1\cong G_2$.
	\end{proof}
	\section{A characterisation of small profinite groups} \label{small groups}
	A key ingredient in our proof of Theorem \ref{Main Theorem} is Segal's characterisation of small profinite groups, Theorem \ref{segal characterization of small groups}. In the proof of Proposition \ref{f.g groups are cancellable} we use another property of small profinite groups - the fact that they are determined by their finite quotients. Our final result is that this property is in fact equivalent to being small. For a profinite group $G$, by $\operatorname{fin}(G)$ we denote set of isomorphism classes of the images $G/U$ as $U$ varies over the open normal subgroups of $G$.
	\begin{thm} \label{small}
		A profinite group $G$ is small if and only if $G$ satisfies the following property: Whenever $H$ is a profinite group such that $\operatorname{fin}(G)=\operatorname{fin}(G)$, then $G\cong H$.
	\end{thm} 
	\begin{proof}
		$\Rightarrow$ As mentioned above, this is precisely \cite[Proposition 16.10.7]{fried2008field}.
		
		$\Leftarrow$ For the converse suppose that $G$ is not small. Then as shown in \cite{helbig} there exist open normal subgroups $G_0>M_0$ of $G$ such that $T:=G_0/M_0$ is a minimal normal subgroup of $G/M_0$ and moreover there is an infinite collection $(M_i)_{i=0}^\infty$ of open normal subgroups of $G$ with $M_i <G_0$ and $G/M_i$ isomorphic to $G/M_0$ via an isomorphism inducing the identity on $G/G_0$. Let $L= \cap_{i=0}^\infty M_i$ and note that $G_0/L \cong T^\mathbb N$ is either a semisimple group or a semisimple $G/G_0$-module which is a Cartesian power of the simple $G/G_0$-module $T$. Note that in both cases any $G$-invariant subgroup or quotient of $L$ is also a Cartesian power of $T$.
		
		For a set $J$ finite or infinite we define
		\[ G_J:= \{ (g,(g_j)_{j \in J}) \in G \times (G/M_0)^J \ | \ g_j \equiv g \textrm{ mod } G_0 \quad \forall j \in J \}. \]
		
		The Theorem will follow from 
		\begin{prop} \label{pr} The groups $G$ and $G_J$ have the same set of finite images.
		\end{prop}
		\begin{proof} It is enough to show $\mathrm{Im}(G_J) \subseteq \mathrm{Im}(G)$ since the reverse inclusion is clear. Every finite topological image of $G_J$ s an image of 
			\[ \bar G_{n,N}:= \{ (gN ,(g_j)_{j=1}^n) \in G/N \times (G/M_0)^n \ | \ g_j \equiv g \textrm{ mod } G_0 \quad j=1, \ldots, n \}, \]
			for some integer $n$ and some open normal subgroup $N$ of $G$ with $N \leq G_0$.  It is sufficient to show that $\bar G_{n,N}$ occurs as a finite topological quotient of $G$. 
			
			Consider the the open normal subgroup $NL$ of $G$. We have $G_0 \geq NL > L$ with $G_0/L = T^\mathbb N$, therefore  $NL/L$ is an infinite product of copies of $T$ and hence there exists an open normal subgroup $N_1$ of $G$ with $N_1<N$ and $NL/N_1L \cong T^m$ for some $m \geq n$. We can replace $G$ with its finite image $G/N_1$ and from now on assume that $N_1=1$ and $G_0/L \cong T^s$ and $NL/L \cong T^m$ for some integer $s \geq m$. The semisimplicity of $G_0/L$ implies that $NL/L$ has a $G$-invariant complement $K/L$ in $G_0/L$ for some normal subgroup $G_0 \geq K \geq L$ of $G$ with $KN=G_0$ and $K \cap NL = L$. Again, replacing $G$ by $G/(L \cap N)$ we may assume from now on that $L \cap N=\{1\}$.
			Note that $G_0/K \cong LN/L \cong N \cong T^m$ and therefore \begin{equation} \label{K}  G/K \cong \{ (g_i)_{i=1}^m \in (G/M_0)^m \ | \ g_iG_0=g_jG_0 \ \forall \ 1 \leq i, j \leq m  \}\end{equation}
			
			We claim that $G_0 \cong K \times N$.
			Indeed $KN=G_0$ follows from the definition of $K$, whereas $K \cap N \leq (K \cap LN) \cap N\leq L \cap N=1$. 
			
			Let \[ f: G \rightarrow (G/N) \times (G/K) \] be the homomorphism $f(g)=(gN,gK)$ for all $g \in G$. Since $K \cap N=1$ we have that $f$ is injective and in view of (\ref{K}) and $|G|=|G:N||N|=|G:N||T|^m$ we deduce that image of $f(G)$ is precisely the group $\bar G_{m,N}$. Since $\bar G_{m,N}$ maps onto $\bar G_{n,N}$ for all $n \leq m$ Proposition \ref{pr} follows. \end{proof}
		The proof of Theorem \ref{small} is complete. 
		
	\end{proof}

	\bibliographystyle{plain}
	
\end{document}